\documentclass[12pt,leqno]{amsart}
\usepackage{latexsym,amsmath,amssymb,mathrsfs}
\usepackage{graphicx}

\title[Sobolev Embedding]
{Sobolev Embedding of a Sphere Containing An Arbitrary Cantor Set in the image}

\author[P. Haj\l{}asz]{Piotr Haj\l{}asz}
\address{Piotr Hajlasz, Department of Mathematics, University of Pittsburgh, Pittsburgh, PA 15260, USA, {\tt hajlasz@pitt.edu}}
\author[X. Zhou]{Xiaodan Zhou}
\address{Xiaodan Zhou, Department of Mathematics, University of Pittsburgh, Pittsburgh, PA 15260, USA, {\tt xiz78@pitt.edu}}

\thanks{P.H.\ was supported by NSF grant DMS-1161425.}

plus 6pt minus 12pt
\def\eps{\varepsilon}

\def\id{{\rm id\, }}
\def\I{\mathfrak{I}}
\def\i{\mathfrak{i}}
\def\Can{\mathfrak{C}}
\def\c{\mathfrak{c}}

\newtheorem{theorem}{Theorem}
\newtheorem{lemma}[theorem]{Lemma}
\newtheorem{corollary}[theorem]{Corollary}

\def\diam{{\rm diam\,}}

\def\supp{{\rm supp\,}}
\def\loc{{\rm loc}}
\theoremstyle{definition}
\newtheorem{remark}[theorem]{Remark}

\newcommand{\barint}{
\rule[.036in]{.12in}{.009in}\kern-.16in \displaystyle\int }

\newcommand{\barcal}{\mbox{$ \rule[.036in]{.11in}{.007in}\kern-.128in\int $}}
\newcommand{\Sph}{\mathbb S}
\newcommand{\Ball}{\mathbb B}

\newcommand{\bbbn}{\mathbb N}
\newcommand{\bbbr}{\mathbb R}

\def\loc{\operatorname{loc}}
\def\diam{\operatorname{diam}}

\def\supp{\operatorname{supp}}

\def\mvint_#1{\mathchoice
          {\mathop{\vrule width 6pt height 3 pt depth -2.5pt
                  \kern -8pt \intop}\nolimits_{\kern -3pt #1}}%
          {\mathop{\vrule width 5pt height 3 pt depth -2.6pt
                  \kern -6pt \intop}\nolimits_{#1}}%
          {\mathop{\vrule width 5pt height 3 pt depth -2.6pt
                  \kern -6pt \intop}\nolimits_{#1}}%
          {\mathop{\vrule width 5pt height 3 pt depth -2.6pt
                  \kern -6pt \intop}\nolimits_{#1}}}

\numberwithin{theorem}{section} \numberwithin{equation}{section}

\begin{document}

\sloppy

%\dsp

\sloppy

\begin{abstract}
We construct a large class of pathological $n$-dimensional topological spheres in $\bbbr^{n+1}$
by showing that for any Cantor set $C\subset\bbbr^{n+1}$ there is a topological embedding 
$f:\Sph^n\to\bbbr^{n+1}$ of the Sobolev class $W^{1,n}$ whose image contains the Cantor set $C$.
\end{abstract}

\maketitle

\section{Introduction}
In 1924 J. W. Alexander \cite{alexander1}, constructed a homeomorphism $f:\Sph^2\to f(\Sph^2)\subset\bbbr^3$ so that the unbounded component of
$\bbbr^3\setminus f(\Sph^2)$ is not simply connected. In particular it is not homeomorphic to the complement of the
standard ball in $\bbbr^3$. This famous construction known as the {\em Alexander horned sphere} can be easily generalized to higher dimensions.
The aim of this paper is to show that a large class of pathological topological $n$-dimensional spheres, including the Alexander horned sphere,
can be realized as the images of Sobolev $W^{1,n}$ homeomorphisms, each of which being a smooth diffeomorphism outside of a Cantor set in $\Sph^n$ of
Hausdorff dimension zero. For a precise statement, see Theorem~\ref{main}. Here and in what follows $\Sph^n$ will denote the standard unit sphere in $\bbbr^{n+1}$.

The following classical result \cite[Theorem~30.3]{willard} provides a characterization of spaces that are homeomorphic to the ternary Cantor set:
{\em A metric space is homeomorphic to the ternary Cantor set if and only if it is compact, totally disconnected and has no isolated points.}
Recall that the space is {\em totally disconnected} if the only non-empty connected subsets are one-point sets. In what follows by a {\em Cantor set} we 
will mean any subset of Euclidean space that is homeomorphic to the ternary Cantor set.

Recall also that the {\em Sobolev space} $W^{1,p}$ consists of functions in $L^p$ whose distributional gradient is in $L^p$.
By $f\in W^{1,p}(\Sph^n,\bbbr^{n+1})$ we will mean that the components of the mapping $f:\Sph^n\to\bbbr^{n+1}$ are in $W^{1,p}$.
For more details regarding Sobolev spaces see \cite{aubin,EG}. 

Throughout the paper by an {\em embedding} we will mean a homeomorphism onto the image i.e., $f:X\to Y$ is an embedding if $f:X\to f(X)$ is
a homeomorphism. In the literature such an embedding is often called a topological embedding.
The main result of the paper reads as follows.
\begin{theorem}
\label{main}
For any Cantor set $C\subset\bbbr^{n+1}$, $n\geq 2$, there is an embedding $f:\Sph^n\to \bbbr^{n+1}$ such that
\begin{itemize}
\item[(a)] $f\in W^{1,n}(\Sph^n,\bbbr^{n+1})$,
\item[(b)] $C\subset f(\Sph^n)$,
\item[(c)] $f^{-1}(C)\subset\Sph^n$ is a Cantor set of Hausdorff dimension zero,
\item[(d)] $f$ is a smooth diffeomorphism in $\Sph^n\setminus f^{-1}(C)$.
\end{itemize}
\end{theorem}
\begin{remark}
Our construction resembles that of the Alexander horned sphere and it will be clear that it can be used to construct a version of the
Alexander horned sphere $f:\Sph^n\to f(\Sph^n)\subset\bbbr^{n+1}$ so that $f\in W^{1,n}$ and $f$ is a smooth diffeomorphism outside 
a Cantor set of Hausdorff dimension zero.
\end{remark}
\begin{remark}
A similar technique to the one used in the proof of Theorem~\ref{main} has also been employed in a variety of different settings \cite{besicovitch1,besicovitch2,cesari,goldsteinh,hajlaszt1,hajlaszt2,iwaniecm1,iwaniecm2,kaufman,malym,ponomarev1,ponomarev2,wengery,wildrickz1,wildrickz2}.
\end{remark}
\begin{remark}
According to Theorem~\ref{main} we can construct a topological sphere in $\bbbr^3$ that is $W^{1,2}$ homeomorphic to $\Sph^2$ and that
contains Antoine's necklace. {\em Antoine's necklace} is a Cantor set in $\bbbr^3$ whose complement is not simply connected. Hence the unbounded component of
$\bbbr^3\setminus f(\Sph^2)$ is also not simply connected. This gives a different example with the same topological consequences as those of the
Alexander horned sphere. In fact, using results of Sher \cite{sher} we will show that there are uncountably many ``essentially different'' examples. 
For a precise statement see Theorem~\ref{main2}.
\end{remark}
\begin{remark}
One cannot in general demand the function constructed in the theorem to be in $W^{1,p}$, $p>n$. Indeed, if $f\in W^{1,p}(\Sph^n,\bbbr^{n+1})$, then the image $f(\Sph^n)$
has finite $n$-dimension Hausdorff measure, but a Cantor set in $\bbbr^{n+1}$ may have positive $(n+1)$-dimensional measure and in that case it cannot be contained
in the image of $f$. The fact that the image $f(\Sph^n)$ has finite $n$-dimensional measure follows from the area formula and the integrability of the Jacobian of $f$.
The fact that the area formula is satisfied for mappings $f\in W^{1,p}(\Sph^n,\bbbr^{n+1})$ with $p>n$ is well known and follows from the following 
observations. The area formula is true for Lipschitz mappings \cite[Theorem~3.3.2]{EG}. The domain $\Sph^n$ is the union of countably many sets on which 
$f$ is Lipschitz continuous \cite[Section~6.6.3]{EG} plus a set of measure zero. On Lipschitz pieces the area formula is satisfied. Since the mapping 
$f$ maps sets of measure zero to sets of $n$-dimensional Hausdorff measure zero, \cite[Theorem~4.2]{henclk}, the area formula is in fact true for $f$.
The proof presented in \cite[Theorem~4.2]{henclk} is in the case of mappings into $\bbbr^{n}$, but the same proof works in the case of mappings into $\bbbr^{n+1}$.
\end{remark}

\begin{remark}
It is well known, \cite[Theorem~4.9]{henclk}, that any homeomorphism $f:\Omega\to f(\Omega)\subset\bbbr^n$ of class $W^{1,n}$,
where $\Omega\subset\bbbr^n$ is open, has the {\em Lusin property}, i.e. it maps sets of measure zero to sets of measure zero.
Reshetnyak \cite{reshetnyak} observed that this is no longer true for embeddings
$f\in W^{1,n}(\Sph^n,\bbbr^m)$, when $m>n\geq 2$. In his example he considered $n=2$ and $m=3$, see also
\cite[Example~5.1]{chm}. Later V\"ais\"al\"a \cite{vaisala} generalized it to any $n\geq 2$ and $m>n$.
In the constructions of Reshetnyak and V\"ais\"al\"a a set of measure zero is mapped to a set of positive $n$-dimensional Hausdorff measure.
Theorem~\ref{main} also provides an example of this type. Indeed, if a Cantor set $C\subset\bbbr^{n+1}$ has positive $(n+1)$-dimensional measure,
then the embedding $f\in W^{1,n}(\Sph^n,\bbbr^{n+1})$ from Theorem~\ref{main} maps the set $f^{-1}(C)$ of Hausdorff dimension zero onto the set
$C\subset f(\Sph^n)$ of positive $(n+1)$-dimensional measure.
Actually, in a context of the Lebesgue area a similar example has already been constructed by Besicovitch \cite{besicovitch1,besicovitch2}, but
Besicovitch did not consider the Sobolev regularity of the mapping. The construction of Besicovitch is very different from that of Reshetnyak and V\"ais\"ala and it is more related to ours.
While Besicovitch's construction deals with a particular Cantor set,
we deal with {\em any} Cantor set and we prove that the resulting mapping $f$ belongs to the Sobolev space $W^{1,n}$.
\end{remark}

We say that two embeddings $f,g:\Sph^n\to\bbbr^{n+1}$, are {\em equivalent} if there is a homeomorphism $h:\bbbr^{n+1}\to\bbbr^{n+1}$ such that
$h(f(\Sph^n))=g(\Sph^n)$. 

The second main result of the paper reads as follows.
\begin{theorem}
\label{main2}
There are uncountably many embeddings $f:\Sph^2\to\bbbr^{3}$ of class $W^{1,2}(\Sph^2,\bbbr^{3})$ which are not equivalent.
\end{theorem}

The theorem can be generalized to higher dimensions, but we consider the case $n=2$ only because our proof is based on a result of Sher \cite{sher}
about Cantor sets in $\bbbr^3$. Generalization of Theorem~\ref{main2} to $n\geq 3$ would require a generalization of Sher's result to higher dimensions. Since 
this would be a work of purely technical nature with predicted answer, we do not find it particularly interesting.

The paper is organized as follows. Sections~\ref{cantor} and~\ref{tentacles} contain preliminary material needed in the proof of Theorem~\ref{main}.
In Section~\ref{cantor} we introduce some terminology regarding Cantor sets and we construct a Cantor tree. In Section~\ref{tentacles}
we define surfaces, called tentacles, around smooth curves. Finally in Sections~\ref{theproof} and ~\ref{theproof2} we prove Theorems~\ref{main}
and~\ref{main2} respectively.

\noindent
{\bf Acknowledgements.}
The authors would like to thank the referee for valuable comments that led to an improvement of the paper.

\section{Cantor sets and trees}
\label{cantor}

\subsection{Ternary Cantor set}
The ternary Cantor set will be denoted by $\Can$. It is constructed by removing the middle third of the unit interval $[0,1]$, 
and then successively deleting the middle third of each resulting subinterval. 
Denote by $\I^k$ all binary numbers $i_1 \ldots i_k$ such that $i_j\in \{0,1\}$ for $j=1,2,\ldots,k$, and by $\I^\infty$ all binary infinite
sequences $i_1i_2\ldots$
Clearly, the ternary Cantor set $\Can$ can be written as 
$$
\Can=\bigcap_{k=1}^{\infty}\ \bigcup_{i_1\ldots i_k\in \I^k} I_{i_1\ldots i_k},
$$
where $I_{i_1\ldots i_k}$ is one of the $2^k$ closed intervals in the $k$-th level of the construction of the Cantor set $\Can$; 
the binary number $i_1\ldots i_k$ denotes the position of this interval: 
If $i_k=0$, it is the left subinterval of $I_{i_1\ldots i_{k-1}}$, otherwise it is the right subinterval.

We also have that for any $k\in\bbbn$
$$
\Can=\bigcup_{i_1\ldots i_k\in \I^k} \Can_{i_1\ldots i_k}
$$
where $\Can_{i_1\ldots i_k}=\Can\cap I_{i_1\ldots i_k}$.

Note that points in the Cantor set $\Can$ can be uniquely encoded by infinite binary sequences. Indeed, if $\i=i_1i_2\ldots\in\I^\infty$,
then
$$
\{ \c_{\i}\} =\bigcap_{k=1}^\infty I_{i_1\ldots i_k} = \bigcap_{k=1}^\infty \Can_{i_1\ldots i_k}
$$
consists of a single point $\c_{\i}\in\Can$ and 
$$
\Can =\bigcup_{\i\in\I^\infty} \{ \c_{\i} \}.
$$

\subsection{Cantor trees}
Let $C\subset\bbbr^{n+1}$ be a Cantor set and let $f:\Can\to C$ be a homeomorphism. We will write
$C_{i_1\ldots i_k}=f(\Can_{i_1\ldots i_k})$ and $c_{\i}=f(\c_{\i})$ for $\i\in\I^\infty$. Since the mapping $f$ is uniformly continuous,
\begin{equation}
\label{e4}
\max_{i_1\ldots i_k\in \I^k} (\diam C_{i_1\ldots i_k})\to 0
\quad
\text{as $k\to\infty$.}
\end{equation}
For each $k$ and each $i_1\ldots i_k\in \I^k$ we select a point $A_{i_1\ldots i_k}$ such that
\begin{itemize}
\item The point $A_{i_1\ldots i_k}$ does not belong to $C$,
\item The distance of the point $A_{i_1\ldots i_k}$ to $C_{i_1\ldots i_k}$ is less than $2^{-k}$,
\item $A_{i_1\ldots i_k}\neq A_{j_1\ldots j_\ell}$ if $i_1\ldots i_k\neq j_1\ldots j_\ell$.
\end{itemize}
It is easy to see that if $\i=i_1 i_2\ldots\in\I^\infty$, then
$$
A_{i_1\ldots i_k}\to c_{\i}=f(\c_{\i})
\quad
\text{as $k\to\infty$.}
$$
Indeed, $c_{\i}\in C_{i_1\ldots i_k}$ so
$$
|A_{i_1\ldots i_k}-c_{\i}|<2^{-k}+\diam C_{i_1\ldots i_k}\to 0
\quad
\text{as $k\to\infty$.}
$$
Now we are ready to build a {\em Cantor tree} by adding branches $J_{i_1\ldots i_k i_{k+1}}$ connecting
$A_{i_1\ldots i_k}$ to $A_{i_1\ldots i_k i_{k+1}}$. The precise construction goes as follows.

By translating the coordinate system we may assume that the distance between the origin in $\bbbr^{n+1}$
and the Cantor set $C$ is greater than $100$ (that is way too much, but there is nothing wrong with being generous).

Let $J_0$ and $J_1$ be smooth Jordan arcs (i.e., smoothly embedded arcs without self-intersections) of unit speed (i.e.,
parametrized by arc-length) connecting the origin $0$ to the points $A_0$ and $A_1$ respectively.
We also assume that
\begin{itemize}
\item The curve $J_0$ does not intersect with the curve $J_1$ (except for the common endpoint $0$).
\item The curves $J_0$ and $J_1$ avoid the Cantor set $C$.
\item The curves $J_0$ and $J_1$ meet the unit ball 
$\Ball^n(0,1)\subset\bbbr^n\times\{ 0\}\subset\bbbr^{n+1}$ lying in the hyperplane of the first $n$ coordinates only at the origin
and both curves exit $\Ball^n(0,1)$ on the same side of $\Ball^n(0,1)$ in $\bbbr^{n+1}$.
\end{itemize}
The curves $J_0$ and $J_1$ will be called {\em branches of order $1$}.

A simple topological observation is needed here. While the topological structure of a Cantor set $C$ inside $\bbbr^{n+1}$ may be
very complicated (think of Antoine's necklace), no Cantor set can separate open sets in $\bbbr^{n+1}$. Indeed,
by \cite[Corollary~2 of Theorem~IV~3]{HW} compact sets separating open sets in $\bbbr^{n+1}$ must have topological dimension at least $n$ 
but the topological dimension of a Cantor set is $0$, \cite[Example~II~3]{HW}.
Hence we can connect points in the complement of a Cantor set by smooth Jordan arcs that avoid the Cantor set. 
Moreover we can construct such an arc in a way that it is arbitrarily close to the line segment connecting the endpoints. 

Recall that the {\em $\eps$-neighborhood} of a set $A$ is the set of all points whose distance to the set $A$ is less than $\eps$.

Suppose that we have already constructed all branches $J_{i_1\ldots i_k}$ of order $k\geq 1$.  
The construction of the branches of order $k+1$ goes as follows.
$\{ J_{i_1\ldots i_{k+1}}\}$ is a family of $2^{k+1}$ curves such that
\begin{itemize}
\item $J_{i_1\ldots i_{k+1}}$ is a smooth Jordan arc parametrized by arc-length that connects $A_{i_1\ldots i_k}$
(an endpoint of the branch $J_{i_1\ldots i_{k}}$) to $A_{i_1\ldots i_k i_{k+1}}$.
\item The curves $J_{i_1\ldots i_{k+1}}$ do not intersect with the Cantor set $C$, they do not intersect with each other 
(except for the common endpoints) and
they do not intersect with previously constructed branches of orders less than or equal to $k$ (except for the common endpoints).
\item The image of the curve $J_{i_1\ldots i_{k+1}}$ is contained in the $2^{-k}$-neighborhood of the line segment $\overline{A_{i_1\ldots i_k}A_{i_1\ldots i_{k+1}}}$.
\item The angle between the branch $J_{i_1\ldots i_{k}}$ and each of the emerging branches
$J_{i_1\ldots i_{k}0}$ and $J_{i_1\ldots i_{k}1}$ at the point $A_{i_1\ldots i_k}$ where the curves meet is larger than $\pi/2$.
\end{itemize}
The reason why we require the last condition about the angles is far from being clear at the moment, but it will be clarified in Section~\ref{whatfor}.

In what follows, depending on the situation, $J_{i_1\ldots i_{k}}$ will denote either the curve (a map from an interval to $\bbbr^{n+1}$) or its image
(a subset of $\bbbr^{n+1}$), but it will always be clear from the context what interpretation we use.

A {\em Cantor tree} is the closure of the union of all branches
$$
T=\overline{\bigcup_{k=1}^\infty \bigcup_{i_1\ldots i_k\in \I^k} J_{i_1\ldots i_k}}\, .
$$
We also define $T_k$ to be the tree with branches of orders less than or equal to $k$ removed. Formally
$$
T_k=T\setminus B_k
\quad
\text{where}
\quad 
B_k=\bigcup_{s=1}^k\bigcup_{i_1\ldots i_s\in \I^s} J_{i_1\ldots i_s}.
$$
Note that the set $T_k$ is not closed -- it does not contain the endpoints $A_{i_1\ldots i_k}$.

The branches $J_{i_1\ldots i_{k+1}}$ are very close to the sets $C_{i_1\ldots i_k}$
in the following sense.
\begin{lemma}
\label{l1}
A branch $J_{i_1\ldots i_{k+1}}$ is contained in the
$2^{-k+2}+\diam C_{i_1\ldots i_k}$ neighborhood of $C_{i_1\ldots i_k}$.
\end{lemma}
\begin{proof}
A branch $J_{i_1\ldots i_{k+1}}$ connects the points $A_{i_1\ldots i_k}$ and $A_{i_1\ldots i_{k+1}}$. The distance of
$A_{i_1\ldots i_k}$ and $A_{i_1\ldots i_{k+1}}$ to the set $C_{i_1\ldots i_k}$ is less than $2^{-k}$
(because $C_{i_1\ldots i_{k+1}}\subset C_{i_1\ldots i_k}$). Hence
$$
|A_{i_1\ldots i_k} - A_{i_1\ldots i_{k+1}}|< 2\cdot 2^{-k} + \diam C_{i_1\ldots i_k}
$$
so by the triangle inequality the line segment $\overline{A_{i_1\ldots i_k}A_{i_1\ldots i_{k+1}}}$ is contained in the
$3\cdot 2^{-k}+\diam C_{i_1\ldots i_k}$ neighborhood of $C_{i_1\ldots i_k}$. Since
$J_{i_1\ldots i_{k+1}}$ is contained in the $2^{-k}$ neighborhood of the line segment, the lemma follows.
\end{proof}
\begin{corollary}
\label{c1}
$T_k$ is contained in the
$$
\eps_k:= 2^{-k+2}+ \max_{i_1\ldots i_{k}\in\I^k} (\diam C_{i_1\ldots i_k})
$$ 
neighborhood of the Cantor set $C$ and $\eps_k\to 0$ as $k\to\infty$.
\end{corollary}
\begin{proof}
Indeed, the branches of $T_k$ are of the form $J_{i_1\ldots i_{s+1}}$, $s\geq k$. 
Each such branch is contained in the
$2^{-s+2}+\diam C_{i_1\ldots i_s}$ neighborhood of
$C_{i_1\ldots i_s}\subset C$. Since $s\geq k$ and 
$C_{i_1\ldots i_s}\subset C_{i_1\ldots i_k}$ we have
$$
2^{-s+2}+\diam C_{i_1\ldots i_s}\leq 2^{-k+2}+\diam C_{i_1\ldots i_k}
$$
so $T_k$ is contained in the $\eps_k$ neighborhood of $C$. The fact that $\eps_k \to 0$ follows from 
\eqref{e4}.
The proof is complete.
\end{proof}
Since the sets $B_k$ are compact and their complements $T_k=T\setminus B_k$ are in close proximity of $C$ by
Corollary~\ref{c1}, it easily follows that
$$
T=C\cup\bigcup_{k=1}^\infty \bigcup_{i_1\ldots i_l\in \I^k} J_{i_i\ldots i_k}.
$$

\subsection{What is it for?}
\label{whatfor}

The idea of the proof of Theorem~\ref{main} is to build a surface that looks very similar to the tree $T$ with
one dimensional branches $J_{i_1\ldots i_k}$ of the tree $T$ replaced by smooth thin surfaces built around the curves $J_{i_1\ldots i_k}$; such surfaces
will be called {\em tentacles}.
The parametric surface $f:\Sph^n\to \bbbr^{n+1}$ will be constructed as a limit of smooth surfaces $f_k:\Sph^n\to\bbbr^{n+1}$. This sequence will be defined by
induction. In the step $k$ we replace all branches $J_{i_1\ldots i_k}$ by smooth surfaces -- tentacles.
Such surfaces will be very close to the branches $J_{i_1\ldots i_k}$ and they will pass through the endpoints $A_{i_1\ldots i_k}$.
The only place where the set $T_k$ gets close to the branch $J_{i_1\ldots i_k}$ is the endpoint $A_{i_1\ldots i_k}$
where the two branches $J_{i_1\ldots i_k0}$ and $J_{i_1\ldots i_k1}$ of the set $T_k$ emerge. The surface around $J_{i_1\ldots i_k}$,
and passing through the point $A_{i_1\ldots i_k}$ will be orthogonal to the curve $J_{i_1\ldots i_k}$ at the point $A_{i_1\ldots i_k}$.
Since the branches $J_{i_1\ldots i_k0}$ and $J_{i_1\ldots i_k1}$ emerging from that point form angles larger than $\pi/2$ with $J_{i_1\ldots i_k}$
the surface will not intersect the branches $J_{i_1\ldots i_k0}$ and $J_{i_1\ldots i_k1}$. By making the surfaces around $J_{i_1\ldots i_k}$ thin enough 
we can make them disjoint from the set $T_k$ (note that $A_{i_1\ldots i_k}$ does not belong to $T_k$).

\section{Sobolev tentacles}
\label{tentacles}

It is well known and easy to prove
that
$\eta(x)=\log|\log|x||\in W^{1,n}(\Ball^{n}(0,e^{-1}))$.
Define the truncation of $\eta$ between levels $s$ and
$t$, $0<s<t<\infty$ by
$$
\eta_{s}^{t}(x) =
\left\{
\begin{array}{ccc}
t-s    & \mbox{if $\eta(x)\geq t$,}\\
\eta(x)-s &   \mbox{if $s\leq \eta(x) \leq t$}\\
0 & \mbox{if $\eta(x) \leq s$.}
\end{array}
\right.
$$
Fix an arbitrary $\tau>0$. For every $\delta>0$ there is a sufficiently large $s$
such that
$\widetilde{\eta}_{\delta,\tau}:=\eta_{s}^{s+\tau}$ is a Lipschitz function on
$\bbbr^n$ with the properties:
$$
\supp\widetilde{\eta}_{\delta,\tau}\subset \Ball^{n}(0,\delta/2),
$$
$$
\text{$0\leq\widetilde{\eta}_{\delta,\tau}\leq\tau$
and $\widetilde{\eta}_{\delta,\tau}=\tau$ in a neighborhood $\Ball(0,\delta')$ of $0$},
$$
$$
\int_{\bbbr^n}|\nabla \widetilde{\eta}_{\delta,\tau}|^{n} < \delta^n.
$$
The function $\widetilde{\eta}_{\delta,\tau}$ is not smooth because it is
defined as a truncation, however, mollifying
$\widetilde{\eta}_{\delta,\tau}$
gives a smooth function, denoted by
$\eta_{\delta,\tau}$, with the same properties
as those of $\widetilde{\eta}_{\delta,\tau}$ listed above.
In particular
\begin{equation}
\label{e1}
\int_{\bbbr^n}|\nabla \eta_{\delta,\tau}|^{n} < \delta^n.
\end{equation}

The graph of $\eta_{\delta,\tau}$ restricted to the ball $\overline{\Ball}^n(0,\delta)$ is contained in the cylinder
\begin{equation}
\label{e2}
\{ (x_1,\ldots,x_{n+1})|\, x_1^2+\ldots +x_n^2\leq\delta^2,\ 0\leq x_{n+1}\leq\tau\},
\end{equation}
and it forms a slim ``tower'' around the $x_{n+1}$-axis. The function $\eta_{\delta,\tau}$ equals zero in the annulus $\Ball^n(0,\delta)\setminus \Ball^n(0,\delta/2)$ and 
equals $\tau$ in the ball $\Ball^n(0,\delta')$.

Consider now a smooth Jordan arc
$\gamma:[-1,\tau+1]\to\bbbr^{n+1}$ parametrized by arc-length. We want to 
construct a smooth mapping $\gamma_{\delta}:\Ball^n(0,\delta)\to\bbbr^{n+1}$ whose image will be a smooth, thin, tentacle-shaped
surface around the curve $\gamma|_{[0,\tau]}$. To do this we will apply a diffeomorphism $\Phi$ mapping the cylinder
\eqref{e2} onto a neighborhood of the image of the curve $\gamma$. The tentacle-like surface will be the image of the
graph of $\eta_{\delta,\tau}$ under the diffeomorphism $\Phi$.

The construction of a diffeomorphism $\Phi$ follows a standard procedure. Let
$$
v_1,\ldots,v_n:[-1,\tau+1]\to T\bbbr^{n+1}
$$
be a smooth orthonormal basis in the orthogonal complement of the tangent space to the curve $\gamma$, i.e., for every
$t\in [-1,\tau+1]$,
$\langle v_1(t),\ldots,v_n(t),\gamma'(t)\rangle$
is a positively oriented orthonormal basis of $T_{\gamma(t)}\bbbr^{n+1}$. Now we define
$$
\Phi(x_1,\ldots,x_{n+1})=\gamma(x_{n+1})+\sum_{i=1}^n x_iv_i(x_{n+1})
\quad
\text{for $x\in\bbbr^{n+1}$ with $-1\leq x_{n+1}\leq\tau+1$.}
$$ 
Clearly, $\Phi$ is smooth and its Jacobian equals $1$ along the $x_{n+1}$ axis, $-1< x_{n+1}<\tau+1$.
Hence $\Phi$ is a diffeomorphism in a neighborhood of any point on the $x_{n+1}$-axis, $-1< x_{n+1}< \tau+1$.
Using compactness of the image of the curve $\gamma$ it easily follows that there is a $\delta_0>0$ such that for all $0<\delta<\delta_0$, $\Phi$ is
a diffeomorphism in an open neighborhood of the cylinder \eqref{e2}. Now we define
$$
\gamma_{\delta}:\overline{\Ball}^n(0,\delta)\to\bbbr^{n+1},
\quad
\gamma_{\delta}(x_1,\ldots,x_n)=\Phi(x_1,\ldots,x_{n},\eta_{\delta,\tau}(x_1,\ldots,x_n)).
$$
Since
$$
\frac{\partial\gamma_\delta}{\partial x_i} =
\frac{\partial\Phi}{\partial x_i}+\frac{\partial\Phi}{\partial x_{n+1}}\, \frac{\partial\eta_{\delta,\tau}}{\partial x_i}
$$
it follows that
$$
|D\gamma_\delta|\leq \sqrt{n}\Vert D\Phi\Vert_\infty(1+|\nabla\eta_{\delta,\tau}|),
$$
where $\Vert D\Phi\Vert_\infty$ is the supremum of the Hilbert-Schmidt norms $|D\Phi|$ over the cylinder \eqref{e2}. Hence using \eqref{e1},
for every $\eps>0$ we can find $\delta>0$ so small that
\begin{equation}
\label{e3}
\int_{\Ball^n(0,\delta)} |D\gamma_\delta|^n\leq C(n)\Vert D\Phi\Vert_\infty^n\delta^n<\eps.
\end{equation}
Observe that $\delta$ depends on $\gamma$ (because $\Vert D\Phi\Vert_\infty$ depends on $\gamma$).

The tentacle $\gamma_\delta$ maps the annulus $\overline{\Ball}^n(0,\delta)\setminus \Ball^n(0,\delta/2)$ onto the isometric annulus in the 
hyperplane orthogonal to $\gamma$ at $\gamma(0)$. Indeed, for $x\in\overline{\Ball}^n(0,\delta)\setminus \Ball^n(0,\delta/2)$, 
$\eta_{\delta,\tau}(x)=0$ and hence
$$
\gamma_\delta(x)=\Phi(x_1,\ldots,x_n,0)=\gamma(0)+\sum_{i=1}^n x_iv_i(0)
$$
is an affine isometry.
For a similar reason $\gamma_\delta$ maps the ball $\overline{\Ball}^n(0,\delta')$ onto the isometric ball in the hyperplane
orthogonal to $\gamma$ at $\gamma(\tau)$. Finally for $x\in \overline{\Ball}^n(0,\delta/2)\setminus \Ball^n(0,\delta')$,
$\gamma_\delta$ creates a smooth thin surface around the curve $\gamma$ that connects the annulus at $\gamma(0)$
with the ball at $\gamma(\tau)$.

Composition with a translation allows us to define a tentacle $\gamma_\delta:\overline{\Ball}^n(p,\delta)\to\bbbr^{n+1}$
centered at any point $p\in\bbbr^n$.

\section{Proof of Theorem~\ref{main}}
\label{theproof}

We will construct the Sobolev embedded surface $f:\Sph^n\to \bbbr^{n+1}$ as a limit of smooth embedded surfaces
$f_k:\Sph^n\to \bbbr^{n+1}$. 

By replacing $\Sph^n$ with a diffeomorphic submanifold (still denoted by $\Sph^n$) we may assume that it contains the unit ball
$$
\Ball^n=\Ball^n(0,1)\subset\bbbr^n\times \{ 0\}\subset\bbbr^{n+1}.
$$
lying in the hyperplane of the first $n$ coordinates.

Since the distance of the Cantor set $C$ to the origin is larger than $100$, the only parts of the Cantor tree $T$ that are close to $\Ball^n$ are the branches 
$J_0$ and $J_1$ that connect the origin to $A_0$ and $A_1$. Since the branches meet $\Ball^n$ only at the origin 
and leave $\Ball^n$ on the same side of $\Ball^n$,
we can assume that $\Sph^n$ meets $T$ only at the origin.

Now we will describe the construction of $f_1$. As indicated in Section~\ref{whatfor} we want to grow two tentacles 
from $\Sph^n$ near the branches $J_0$ and $J_1$ all the way to points $A_0$ and $A_1$, but we want to make sure that the tentacles
do not touch the set $T_1$.

To do this we choose two distinct points $p_0,p_1\in\Ball^n$ close to the origin and modify the curves $J_0$ and $J_1$ only near the origin, so that the modified
Jordan arcs $\gamma^0$ and $\gamma^1$ emerge from the points $p_0$ and $p_1$ instead of the origin, and they are orthogonal to $\Ball^n$
at the points $p_0$ and $p_1$. The curves $\gamma^0$ and $\gamma^1$ quickly meet with $J_0$ and $J_1$ and from the points where they meet they coincide with 
$J_0$ and $J_1$, so all non-intersection properties of the curves are preserved. Since the curves are modified only at their beginnings, outside that place
they are identical with $J_0$ and $J_1$.

Next, we find $\delta_1>0$ so small that the balls $\overline{\Ball}^n(p_0,\delta_1)$ and
$\overline{\Ball}^n(p_1,\delta_1)$ are disjoint and contained in $\Ball^n$ and that there are disjoint tentacles
$$
\gamma_{\delta_1}^i:\overline{\Ball}^n(p_i,\delta_1)\to\bbbr^{n+1}
\quad
\text{for $i=0,1$}
$$
along the curves $\gamma^0$ and $\gamma^1$ such that
$$
\int_{\Ball(p_i,\delta_1)}|D\gamma_{\delta_1}^i|^n< 4^{-n}
\quad
\text{for $i=0,1$.}
$$
Observe that $\gamma_{\delta_1}^i(p_i)=A_i$ for $i=0,1$.

Note that the images of small balls $\Ball^n(p_0,\delta_1')$ and $\Ball^n(p_1,\delta_1')$
are isometric balls, orthogonal to the curves $\gamma^0$ and $\gamma^1$ (and hence to the curves
$J_0$ and $J_1$) at the endpoints $A_0$ and $A_1$. Since the branches $J_{00}$, $J_{01}$ form
angles larger than $\pi/2$ with the curve $\gamma^0$ at $A_0$ and the branches $J_{10}$, $J_{11}$
form angles larger than $\pi/2$ with the curve $\gamma^1$  at $A_1$ we may guarantee, by making the
tentacles sufficiently thin, that they are disjoint from the set $T_1$ (observe that the endpoints
$A_0$ and $A_1$ {\em do not} belong to $T_1$). 

Also each annulus $\overline{\Ball}^n(p_i,\delta_1)\setminus\Ball^n(p_i,\delta_1/2)$ for $i=0,1$
is mapped isometrically by $\gamma^i_{\delta_1}$ onto an annulus centered at $\gamma^i(0)=p_i$
in the hyperplane orthogonal to $\gamma^i$ at $\gamma^i(0)=p_i$. Since the curve $\gamma^i$ is 
orthogonal to $\Ball^n$ at $\gamma^i(0)=p_i$, the annulus
$\overline{\Ball}^n(p_i,\delta_1)\setminus\Ball^n(p_i,\delta_1/2)$
is mapped isometrically onto itself. By choosing an appropriate orthonormal frame
$v_1^i,\ldots v_n^i$ in the definition of $\gamma_{\delta_1}^i$ we may assume that
$\gamma_{\delta_1}^i$ is the identity in the annulus. This guarantees that the mapping 
$$
f_1(x) =
\left\{
\begin{array}{ccc}
\gamma_{\delta_1}^0(x)    & \mbox{if} & x\in\overline{\Ball}^n(p_0,\delta_1),\\
\gamma_{\delta_1}^1(x)    & \mbox{if} & x\in\overline{\Ball}^n(p_1,\delta_1),\\
x                    & \mbox{if} 		 & x\in\Sph^n\setminus(\Ball^n(p_0,\delta_1)\cup \Ball^n(p_1,\delta_1))
\end{array}
\right.
$$
is continuous and hence smooth.
The construction guarantees also that $f_1$ is a smooth embedding of $\Sph^n$ into $\bbbr^{n+1}$
with the image that is disjoint from $T_1$.

The mapping $f_1$ maps the small balls $\overline{\Ball}^n(p_0,\delta_1')$ and
$\overline{\Ball}^n(p_1,\delta_1')$ onto isometric balls centered at $A_0$ and $A_1$ respectively with $f_1(p_i)=A_i$ for $i=0,1$. 
Now the mapping $f_2$ will be obtained from $f_1$ by adding four more tentacles: from the ball $f_1(\Ball^n(p_0,\delta_1'))$ centered at $A_0$ there will be two
tentacles connecting this ball to the points $A_{00}$ and $A_{01}$ and from the ball at $f_1(\Ball^n(p_1,\delta_1'))$ centered at $A_1$ there will be two
tentacles connecting this ball to the points $A_{10}$ and $A_{11}$. More precisely the inductive step is described as follows.

Suppose that we have already constructed a mapping $f_k$, $k\geq 1$ such that
\begin{itemize}
\item $f_k$ is a smooth embedding of $\Sph^n$ into $\bbbr^{n+1}$ whose image is disjoint from $T_k$.
\item There are $2^k$ disjoint balls 
$\overline{\Ball}^n(p_{i_1\ldots i_k},\delta_k)\subset \Ball^n$ and $2^k$
tentacles
$$
\gamma_{\delta_k}^{i_1\ldots i_k}:\overline{\Ball}^n(p_{i_1\ldots i_k},\delta_k)\to\bbbr^{n+1}
$$
for $i_1\ldots i_k\in \I^k$ such that
\begin{itemize}
\item[$\Diamond$]
$\gamma_{\delta_k}^{i_1\ldots i_k}(p_{i_1\ldots i_k})=A_{i_1\ldots i_k}.$
\item[$\Diamond$]
The image of $\gamma^{i_1\ldots i_k}_{\delta_k}$ is in the $2^{-k}$ neighborhood of the curve 
$J_{i_1\ldots i_k}$.
\item[$\Diamond$] We have
\begin{equation}
\label{e5}
\int_{\Ball^n(p_{i_1\ldots i_k},\delta_k)}|D\gamma_{\delta_k}^{i_1\ldots i_k}|^n\, dx < 4^{-nk}.
\end{equation}
\end{itemize}
\item The mapping $f_k$ satisfies
$$
f_k =
\left\{
\begin{array}{ccc}
\gamma_{\delta_k}^{i_1\ldots i_k}    & \mbox{in} & \mbox{$\overline{\Ball}^n(p_{i_1\ldots i_k},\delta_k)$  for $i_1\ldots i_k\in \I^k$},\\
f_{k-1}                              & \mbox{in} &\Sph^n\setminus\bigcup_{i_1\ldots i_k\in \I^k}\Ball^n(p_{i_1\ldots i_k},\delta_k).
\end{array}
\right.
$$
\end{itemize}
Observe that $f_k(p_{i_1\ldots i_k})=A_{i_1\ldots i_k}$ and that for some small $0<\delta_k'<\delta_k$,
$f_k$ maps balls $\overline{\Ball}^n(p_{i_1\ldots i_k},\delta_k')$ onto isometric balls centered at $A_{i_1\ldots i_k}$.

Now we will describe the construction of the mapping $f_{k+1}$.

For each $i_1\ldots i_k\in \I^k$ we choose two points
$$
p_{i_1\ldots i_k 0}, p_{i_1\ldots i_k 1} \in \Ball^n(p_{i_1\ldots i_k},\delta_k')
$$
and modify the curves $J_{i_1\ldots i_k 0}$ and $J_{i_1\ldots i_k 1}$ to 
$\gamma^{i_1\ldots i_k 0}$ and $\gamma^{i_1\ldots i_k 1}$ in a pretty similar way as we did for the curves 
$\gamma^0$ and $\gamma^1$: the new curves $\gamma^{i_1\ldots i_k 0}$ and $\gamma^{i_1\ldots i_k 1}$
emerge from the points $f_k(p_{i_1\ldots i_k 0})$ and $f_k(p_{i_1\ldots i_k 1})$, they are orthogonal to the ball
$f_k(\Ball^n(p_{i_1\ldots i_k},\delta_k'))$ at these points and then they quickly meet and coincide with 
$J_{i_1\ldots i_k 0}$ and $J_{i_1\ldots i_k 1}$.

We find $\delta_{k+1}>0$ so small that
\begin{itemize}
\item The balls
$$
\overline{\Ball}^n(p_{i_1\ldots i_k 0},\delta_{k+1}), \overline{\Ball}^n(p_{i_1\ldots i_k 1},\delta_{k+1}) \subset \Ball^n(p_{i_1\ldots i_k},\delta_k')
$$
are disjoint.
\item There are tentacles
$$
\gamma^{i_1\ldots i_k i}_{\delta_{k+1}}:\overline{\Ball}^n(p_{i_1\ldots i_k i},\delta_{k+1})\to\bbbr^{n+1}
\quad
\text{for $i=0,1$}
$$
such that
\begin{itemize}
\item[$\Diamond$]
$\gamma^{i_1\ldots i_k i}_{\delta_{k+1}}(p_{i_1\ldots i_k i})=A_{i_1\ldots i_k i}$.
\item[$\Diamond$]
The image of $\gamma^{i_1\ldots i_k i}_{\delta_{k+1}}$ is in the $2^{-(k+1)}$ neighborhood of the curve 
$J_{i_1\ldots i_k i}$.
\item[$\Diamond$]
The tentacles do not intersect and they avoid the set $T_{k+1}$ 
\item[$\Diamond$] 
We have
$$
\int_{\Ball^n(p_{i_1\ldots i_k i},\delta_{k+1})} |D\gamma^{i_1\ldots i_k i}_{\delta_{k+1}}|^n\, dx < 4^{-n(k+1)}.
$$
\end{itemize}
\end{itemize}
The condition about the distance of the tentacle to the curve $J_{i_1\ldots i_k i}$ can be easily guaranteed, because
the curve $\gamma^{i_1\ldots i_k i}$ can be arbitrarily close to $J_{i_1\ldots i_k i}$ and the tentacle can be arbitrarily
thin.

Because the curves $\gamma^{i_1\ldots i_k i}$, $i=0,1$, are orthogonal to the balls
$f_k(\Ball^n(p_{i_1\ldots i_k},\delta_k'))$ at the points
$\gamma^{i_1\ldots i_k i}(0)=f_k(p_{i_1\ldots i_k i})$,
by choosing appropriate orthonormal frames in the definition of 
$\gamma^{i_1\ldots i_k i}_{\delta_{k+1}}$ we may guarantee one more condition
\begin{itemize}
\item 
$\gamma^{i_1\ldots i_k i}_{\delta_{k+1}} = f_k$
in
$\overline{\Ball}^n(p_{i_1\ldots i_k i},\delta_{k+1})\setminus\Ball^n(p_{i_1\ldots i_k i},\delta_{k+1}/2)$.
\end{itemize}
We are using here the fact that both $\gamma^{i_1\ldots i_k i}_{\delta_{k+1}}$ and $f_k$ are isometries in that annulus.
Now we define
$$
f_{k+1} =
\left\{
\begin{array}{ccc}
\gamma_{\delta_k}^{i_1\ldots i_k i_{k+1}}    & \mbox{in} & \mbox{$\overline{\Ball}^n(p_{i_1\ldots i_{k+1}},\delta_{k+1})$ for $i_1\ldots i_{k+1}\in \I^{k+1}$},\\
f_{k}                              & \mbox{in} & \Sph^n\setminus\bigcup_{i_1\ldots i_{k+1}\in \I^{k+1}}\Ball^n(p_{i_1\ldots i_{k+1}},\delta_{k+1}).
\end{array}
\right.
$$
As before $f_{k+1}$ is a smooth embedding of $\Sph^n$ into $\bbbr^{n+1}$ whose image is disjoint from $T_{k+1}$.

Let
$$
W_k=\bigcup_{i_1\ldots i_k\in \I^k} \overline{\Ball}^n(p_{i_1\ldots i_k},\delta_k).
$$
Clearly, $W_k$ is a decreasing sequence of compact sets and 
$$
E:=\bigcap_{k=1}^\infty W_k
$$
is a Cantor set $E\subset\Ball^n\subset\Sph^n$. By making the sequence $\delta_k$ converge to zero sufficiently fast we may guarantee that
the Hausdorff dimension of $E$ equals zero.
Similarly as in the case of the ternary Cantor set, points in the set $E$ can be encoded by infinite binary sequences
For $\i=i_1 i_1\ldots \in\I^\infty$ we define
$$
\{ e_{\i}\}=\bigcap_{k=1}^\infty \overline{\Ball}^n(p_{i_1\ldots i_k,\delta_k})
\quad
\text{so}
\quad
E=\bigcup_{\i\in\I^\infty} \{e_{\i}\}.
$$
Now we define
$$
f(x)=
\left\{
\begin{array}{ccc}
x						&  \text{if} 	& x\in\Sph^n\setminus W_1,\\
f_k(x)			&  \text{if}	& x\in W_k\setminus W_{k+1},\ k=1,2,\ldots\\
c_{\i} 			&  \text{if}	& x=e_{\i}\in E,\ \i\in\I^\infty.
\end{array}
\right.
$$
Recall that $c_{\i}=f(\c_{\i})$ is a point of the Cantor set $C$.

\begin{lemma}
\label{l3}
$f=f_k$ when restricted to $\Sph^n\setminus W_{k+1}$.
\end{lemma}
\begin{proof}
The lemma can be easily proved by induction. Let $f_0=\id$.
By the definition of $f$, $f=f_0$ in $\Sph^n\setminus W_1$.
Suppose now that $f=f_k$ in 
$\Sph^n\setminus W_{k+1}$. According to the construction of $f_{k+1}$,
$f_{k+1}=f_k$ in $\Sph^n\setminus W_{k+1}$, but the definition of $f$ yields
$f=f_{k+1}$ in $W_{k+1}\setminus W_{k+2}$ so $f=f_{k+1}$ in
$$
(\Sph^n\setminus W_{k+1})\cup(W_{k+1}\setminus W_{k+2})=\Sph^n\setminus W_{k+2}.
$$
This proves the lemma.
\end{proof}

Since each of the mappings $f_k$ is a smooth embedding whose image 
does not intersect with the Cantor set $C$ it follows from the lemma that $f$ restricted to the open set $\Sph^n\setminus E$
is a smooth embedding and $f(\Sph^n\setminus E)\cap C=\emptyset$. In the remaining Cantor set $E$, the mapping $f$ is defined as a bijection that maps
$E$ onto $C$. Therefore the mapping $f$ is one-to-one in $\Sph^n$ and $C\subset f(\Sph^n)$. 

It remains to prove that $f$ is continuous and that $f\in W^{1,n}$. 

First we will prove that $f\in W^{1,n}$. The mapping $f$ is bounded and hence its components are in $L^n$.
Since the mapping $f$ is smooth outside the Cantor set $E$ of Hausdorff dimension zero, according to the 
characterization of the Sobolev space by absolute continuity on lines, \cite[Section~4.9.2]{EG}, it suffices to show that the classical derivative of $f$
defined outside of $E$ (and hence a.e. in $\Sph^n$) belongs to $L^n(\Sph^n)$.
We have
\begin{eqnarray*}
\int_{\Sph^n}|Df|^n
& = & 
\int_{\Sph^n\setminus E} |Df|^n =
\int_{\Sph^n\setminus W_1} |Df|^n +
\sum_{k=1}^\infty \int_{W_k\setminus W_{k+1}} |Df_k|^n \\
& \leq &
\int_{\Sph^n\setminus W_1} |Df|^n +
\sum_{k=1}^\infty \int_{W_k} |D f_k|^n.
\end{eqnarray*}
Since $f(x)=x$ in $\Sph^n\setminus W_1$, we do not have to worry about the first term on the right hand side and it remains to estimate the infinite sum.

Note that $f_k=\gamma_{\delta_k}^{i_1\ldots i_k}$ in $\overline{\Ball}^n(p_{i_1\ldots i_k},\delta_k)$ so \eqref{e5} yields
$$
\int_{\overline{\Ball}^n(p_{i_1\ldots i_k},\delta_k)} |Df_k|^n =
\int_{\overline{\Ball}^n(p_{i_1\ldots i_k},\delta_k)} |D\gamma^{i_1\ldots i_k}_{\delta_k}|^n < 4^{-nk}.
$$
Hence
$$
\int_{W_k} |Df_k|^n = \sum_{i_1\ldots i_k\in \I^k} 
\int_{\overline{\Ball}^n(p_{i_1\ldots i_k},\delta_k)} |Df_k|^n  < 2^k\cdot 4^{-nk} < 2^{-nk}
$$
so
$$
\sum_{k=1}^\infty \int_{W_k} |D f_k|^n<\sum_{k=1}^\infty 2^{-nk}<\infty.
$$
This completes the proof  that $f\in W^{1,n}$.

\begin{remark}
Replacing the estimate in \eqref{e5} by $\eps 4^{-nk}$ one can easily modify the construction so that the mapping $f$ will have 
an arbitrarily small Sobolev norm $W^{1,n}$.
\end{remark}

It remains to prove that $f$ is continuous. We will need 
\begin{lemma}
\label{l2}
$f(\overline{\Ball}^n(p_{i_1\ldots i_k},\delta_k))$ is contained in the
$$
r_k:= 2^{-k+4}+\diam C_{i_1\ldots i_{k-1}}
$$
neighborhood of $C_{i_1\ldots i_{k-1}}$.
\end{lemma}
\begin{proof}
If $x\in\overline{\Ball}^n(p_{i_1\ldots i_k},\delta_k)\cap E$,
then $x=e_{\i}$ for some $\i\in\I^\infty$ with the first $k$ binary
digits equal $i_1,\ldots,i_k$ i.e.,
$\i=i_1\ldots i_k\ldots$ Hence
$$
f(x)=f(e_{\i}) = c_{\i} =c_{i_1\ldots i_k\ldots}\in C_{i_1\ldots i_{k-1}}.
$$
If $x\in\overline{\Ball}^n(p_{i_1\ldots i_k},\delta_k)\setminus E$,
then there is $s\geq k$ such that
$$
x\in\overline{\Ball}^n(p_{i_1\ldots i_k i_{k+1}\ldots i_s},\delta_s)\setminus W_{s+1}\subset W_s\setminus W_{s+1}
$$
for some binary numbers $i_{k+1},\ldots,i_s$. 

Since $f=f_s$ in $W_s\setminus W_{s+1}$ and
$f_s=\gamma^{i_1\ldots i_s}_{\delta_s}$ in
$\overline{\Ball}^n(p_{i_1\ldots i_s},\delta_s)$ we conclude that
$f(x)=\gamma^{i_1\ldots i_s}_{\delta_s}(x)$.

It remains to show that the image of $\gamma^{i_1\ldots i_s}_{\delta_s}$
is contained in the $r_k$ neighborhood of $C_{i_1\ldots i_{k-1}}$.

By Lemma~\ref{l1}, $J_{i_1\ldots i_s}$ is in the 
$2^{-s+3}+\diam C_{i_1\ldots i_{s-1}}$ neighborhood of $C_{i_1\ldots i_{s-1}}$.
Also the image of $\gamma^{i_1\ldots i_s}_{\delta_s}$ is contained in the $2^{-s}$
neighborhood of $J_{i_1\ldots i_s}$ so the image of
$\gamma^{i_1\ldots i_s}_{\delta_s}$ is contained in the
$$
2^{-s+4}+\diam C_{i_1\ldots i_{s-1}}
$$
neighborhood of $C_{i_1\ldots i_{s-1}}$. Since
$C_{i_1\ldots i_{s-1}}\subset C_{i_1\ldots i_{k-1}}$
and
$$
2^{-s+4}+\diam C_{i_1\ldots i_{s-1}} \leq 2^{-k+4} +\diam C_{i_1\ldots i_{k-1}}
$$
the lemma follows.
\end{proof}

Now we are ready to complete the proof of continuity of $f$. Clearly, $f$ is continuous on $\Sph^n\setminus E$
so it remains to prove its continuity on the Cantor set $E$.
Let $e_{\i}\in E$, $\i=i_1 i_2\ldots\in\I^\infty$. Since 
$f(e_{\i})=c_{\i}$ we need to show that for any $\eps>0$ there is $\delta>0$
such that if $|x-e_{\i}|<\delta$, then $|f(x)-c_{\i}|<\eps$.

Let $\eps>0$ be given. Let $k$ be so large that
$$
2^{-k+4}+ 2\, \diam C_{i_1\ldots i_{k-1}}<\eps
$$
and let $\delta>0$ be so small that
$$
\Ball^n(e_{\i},\delta)\subset \Ball^n(p_{i_1\ldots i_k},\delta_k).
$$
If $|x-e_{\i}|<\delta$, then
$x\in\overline{\Ball}^n(p_{i_1\ldots i_k},\delta_k)$
so by Lemma~\ref{l2}, $f(x)$ belongs to the $r_k$ neighborhood of the set
$C_{i_1\ldots i_{k-1}}$. Since
$c_{\i}\in C_{i_1\ldots i_{k-1}}$, the distance $|f(x)-c_{\i}|$ is less than
$$
r_k+\diam C_{i_1\ldots i_{k-1}} = 2^{-k+4} + 2\, \diam C_{i_1\ldots i_{k-1}}<\eps.
$$
The proof is complete.
\hfill $\Box$

\section{Proof of Theorem~\ref{main2}}
\label{theproof2}

Antoine's necklace $A$ is a Cantor set that is constructed iteratively as follows: Inside a solid torus $A_0$ in $\bbbr^3$ (iteration $0$) we  
construct a chain $A_1$ (iteration $1$) of linked solid tori so that the chain cannot be contracted to a point inside the torus $A_0$. 
$A_1$ is a subset of $\bbbr^3$, the union of the linked tori.
Iteration $n+1$ is obtained from the  iteration $n$ by constructing a chain of tori inside each of the tori of $A_n$ (i.e. inside each of the connected components of $A_n$).
Again $A_{n+1}$ is a subset of $\bbbr^3$ -- the union of all tori in this step of construction.
We also assume that the maximum of the diameters of tori in iteration $n$ converges to zero as $n$ approaches to infinity.
Antoine's necklace is the intersection $A=\bigcap_{n=0}^\infty A_n$.
For more details we refer to \cite{sher}.

Antoine's necklace has the following properties:
\begin{itemize}
\item $\bbbr^3\setminus A$ is not simply connected.
\item For any $x\in A$ and any $r>0$, $A\cap \Ball^3(x,r)$ contains Antoine's necklace.
\end{itemize}
The first property is well known \cite[Chapter~18]{moise} while the second one is quite obvious: $\Ball^3(x,r)$ contains
one of the tori $T$ of one of the iterations (actually infinitely many of such tori) and  $T\cap A$ is also Antoine's necklace because of the iterative
nature of the procedure.

We also need the following observation.

\begin{lemma}
\label{L1} 
Let $A$ be Antoine's necklace and let $M$ be a smooth $2$-dimensional surface in $\bbbr^3$. Then $A\cap M$
is contained in the closure of $A\setminus M$, $A\cap M\subset\overline{A\setminus M}$.
\end{lemma}
\begin{proof}
Suppose to the contrary that for some $x\in A\cap M$, and some $r>0$ we have $\Ball^3(x,r)\cap(A\setminus M)=\emptyset$. By taking $r>0$ sufficiently small we can assume that 
$\Ball^3(x,r)\cap M$ is diffeomorphic to a disc. More precisely, there is a diffeomorphism $\Phi$ of $\bbbr^3$ which maps $\Ball^3(x,r)\cap M$ onto the ball $\Ball^2(0,2)$ in the 
$xy$-coordinate plane. 
Note that by the second property listed above $\Ball^3(x,r)\cap A$ contains 
Antoine's necklace denoted by $\tilde{A}$.  Since $\tilde{A}\cap (A\setminus M)=\emptyset$ we have 
$\tilde{A}\subset \Ball^3(x,r)\cap M$ so $\Phi(\tilde{A})\subset \Ball^2(0,2)$ and $(\Ball^2(0,2)\times\bbbr)\setminus\Phi(\tilde{A})$ is not simply connected.
By \cite[Theorem~13.7 p.93]{moise} there is a homeomorphism $h$ of the ball $\Ball^2(0,2)$ onto itself in such that $\Phi(\tilde{A})$ is mapped
onto the standard ternary Cantor set $\Can$ on the $x$-axis. This homeomorphism can be trivially extended to a homeomorphism of $\Ball^2(0,2)\times\bbbr$ by
letting $H(x,y,z)=(h(x,y),z)$. Clearly the complement of $H(\Phi(\tilde{A}))$ in $\Ball^2(0,2)\times\bbbr$ is not simply connected. On the other hand since
$H(\Phi(\tilde{A}))=\Can$, the complement of this set is simply connected in $\Ball^2(0,2)\times\bbbr$ which is a contradiction. 
\end{proof}

The key argument in our proof is the following result of Sher \cite[Corollary~1]{sher} that we state as a lemma.
\begin{lemma}
\label{L2}
There is an uncountable family of Antoine's necklaces $\{A_i\}_{i\in I}$ such that for
any $i,j\in I$, $i\neq j$ there is no homeomorphism $h:\bbbr^3\to\bbbr^3$ with the property that $h(A_i)=A_j$.
\end{lemma}

For each of the sets $A_i$, let $f_i:\Sph^2\to\bbbr^3$ be an embedding as in Theorem~\ref{main} with the property that $A_i\subset f_i(\Sph^2)$ and $f_i(\Sph^2)\setminus A_i$ is a smooth surface (but not closed). 
It remains to prove that for $i\neq j$ there is no homeomorphism $h:\bbbr^3\to\bbbr^3$ such that $h(f_i(\Sph^2))=f_j(\Sph^2)$. 
Suppose to the contrary that such a homeomorphism $h$ exists. We will show that $h(A_i)=A_j$ which is a contradiction with Lemma~\ref{L2}.

Clearly,
$$
h(A_i)=(h(A_i)\cap A_j)\cup (h(A_i)\cap (f_j(\Sph^2)\setminus A_j)).
$$
Since $f_j(\Sph^2)\setminus A_j$ is a smooth surface, Lemma~\ref{L1} yields
$$
h(A_i)\cap (f_j(\Sph^2)\setminus A_j)\subset
\overline{h(A_i)\setminus (f_j(\Sph^2)\setminus A_j)}=
\overline{h(A_i)\cap A_j}\subset A_j
$$
so $h(A_i)\cap (f_j(\Sph^2)\setminus A_j)=\emptyset$ and hence $h(A_i)\subset A_j$.
Applying the same argument to $h^{-1}$ we obtain that $h^{-1}(A_j)\subset A_i$ so $A_j\subset h(A_i)$ and hence
$h(A_i)=A_j$. The proof is complete.
\hfill $\Box$

\end{document}